\documentclass[12pt]{amsart}
\usepackage{amstext,amsfonts,amssymb,amscd,amsbsy,amsmath,verbatim,fullpage}
\usepackage[alphabetic,lite,backrefs]{amsrefs} 
\usepackage{ifthen}
\usepackage{color,tikz}
\usepackage{amsthm}
\usepackage{latexsym}
\usepackage[all]{xy}
\usepackage{enumerate}

\newtheorem{lemma}{Lemma}[section]

\newtheorem{cor}[lemma]{Corollary}
\newtheorem{conj}[lemma]{Conjecture}

\newtheorem{claim*}{Claim}
\newtheorem{thm}[lemma]{Theorem}
\newtheorem{defn}[lemma]{Definition} 

\newtheorem{question}[lemma]{Question}

\theoremstyle{remark}
\newtheorem{remark}[lemma]{Remark}
\newtheorem{example}[lemma]{Example}

\newcommand{\coker}{\operatorname{coker}}

\newcommand{\im}{\operatorname{im}}

\newcommand{\id}{\operatorname{id}}

\newcommand{\rank}{\operatorname{rank}}
\newcommand{\codim}{\operatorname{codim}}

\newcommand{\GL}{{GL}}

\newcommand{\mf}{\operatorname{MF-rank}}
\newcommand{\mcm}{\operatorname{MCM-rank}}

\newcommand{\defi}[1]{\textsf{#1}} 

\newcommand{\PP}{\mathbb P}

\newcommand{\bR}{\mathbf R}

\newcommand{\ZZ}{\mathbb Z}

\title{Matrix factorizations of generic polynomials}
\author{Daniel Erman}
\address{Department of Mathematics, University of Wisconsin, Madison, WI}
\email{derman@math.wisc.edu}

\begin{document}
\begin{abstract}
We prove that the Buchweitz-Greuel-Schreyer Conjecture on the minimal rank of a matrix factorization holds for a generic polynomial of given degree and strength.  The proof introduces a notion of the secondary strength of a polynomial, and uses a variant of the ultraproduct technique of Erman, Sam, and Snowden.
\end{abstract}

\maketitle
\section{Introduction}

Buchweitz-Greuel-Schreyer have conjectured~\cite[Conjecture A]{bgs}:
 \begin{conj}\label{conj:main}
Let $P$ be a regular local ring with algebraically closed residue field and let $f\in P$.  Let $e:=\lfloor \frac{1}{2}(\codim_{P/f} \operatorname{Sing}(P/f) -2)\rfloor$  where $\operatorname{Sing}(P/f)$ denotes the singular locus of $P/f$.  The minimal rank of a nonfree maximal Cohen-Macaulay module on $P/f$ is $\geq 2^{e}$.
\end{conj}
A \defi{matrix factorization} of $f\in P$ is a pair $(\phi,\psi) = (\phi\colon F\to G, \psi\colon G\to F)$ where $F$ and $G$ are finite rank free $P$-modules, and where $\phi\circ \psi = f\cdot \operatorname{id}_G$ and $\psi\circ \phi =f\cdot \operatorname{id}_F$. 
Based on the equivalence between matrix factorizations of $f$ and maximal Cohen--Macaulay (MCM) modules on the corresponding hypersurface ring~\cite{eisenbud}, this conjecture is closely related to the study of matrix factorizations of $f$.
Matrix factorizations were introduced by Eisenbud in~\cite{eisenbud}, though they arise in Shamash's~\cite{shamash} and related ideas appear in~\cite{G1}.  
Matrix factorizations now arise in connection with a huge range of topics: for a sampling, see~\cite{aspinwall, avramov-infinite, bfk,MP-infinite,orlov} and the discusssion and references within.

Yet despite this wide literature on matrix factorizations, there has been limited progress on Conjecture~\ref{conj:main}.  Since its origin, the conjecture has been known for quadratic polynomials and simple hypersurface singularities~\cite{knorrer}.  Eisenbud, Peeva, and Schreyer have shown that the conjecture holds for certain generic combinations of quadrics~\cite[Corollary~2.2]{eps}, while Ravindra and Tripathi have proven the conjecture for homogeneous polynomials when $e\leq 2$~\cite{tripathi,rt1}; see also~\cite{beh,blaser-eisenbud-schreyer,mrr1,mrr2,rt2,rt3}.  

Conjecture~\ref{conj:main} has both algebraic and geometric motivations.  Buchweitz, Greuel, and Schreyer produced certain explicit matrix factorizations from Koszul complexes~\cite[\S2.1]{bgs}, and their conjecture roughly proposes that these are the smallest possible.  This echoes the Buchsbaum-Eisenbud-Horrocks Conjecture on minimal free resolutions (see~\cite[p. 405]{be} and \cite[Problem 24]{hartshorne-vb})--which  proposes that the Koszul  complex is the smallest possible resolution--as well as related rank conjectures of Carlsson, Avramov-Buchweitz-Iyengar, and Avramov-Buchweitz~\cite{carlsson,AB, ABI}.  In a geometric direction: if $V(f)\subseteq \mathbb P^n$ is a general hypersurface, then $e(f) = \lfloor \frac{n}{2}\rfloor -1$ and the construction from~\cite[\S2.1]{bgs} is built on a nontrivial cycle on the hypersurface; by the Lefschetz Theorem, this cannot happen before the middle codimension.

Many of the conjectures listed above have seen significant recent breakthroughs due to the results of Walker~\cite{walker} and the counterexamples of Iyengar-Walker~\cite{iyengar-walker}.  But Conjecture~\ref{conj:main} has remained untouched by this progress.\footnote{However, in light of the counterexamples of~\cite{iyengar-walker}, Eisenbud, Peeva and Schreyer have used Conjecture~\ref{conj:main} to update the Betti Degree Conjecture of~\cite{AB};  see~\cite[Conjecture~4.2]{eps}.}

We prove that Conjecture~\ref{conj:main} holds for ``generic'' homogeneous polynomials, though we emphasize that our genericity notion is based on a less standard way of parametrizing polynomials.
Instead of fixing the degree and number of variables of $f$, we will fix the degree and the strength of the homogeneous polynomial.  
Recall that the \defi{strength} of $f$ is the minimal $s$ for which we can write $f=\sum_{i=0}^s h_ig_i$ with $1\leq \deg(h_i), \deg(g_i) < \deg(f) - 1$, or $\infty$ if no such decomposition exists.  For example, $x_0y_0+x_1y_1+\cdots + x_sy_s$ has strength $s$.
And while the Buchweitz-Greuel-Schreyer Conjecture was not previously known for any homogeneous polynomial $f$ of degree $d$ with $d>2$ and $e(f) >2$, our results imply that the conjecture holds and is sharp for some polynomial with every $d$ and $e(f)$ value; see Example~\ref{ex:all d and e}.

We define the \defi{rank} of a matrix factorization $(\phi,\psi)$ to be the rank of the free module $F$.  For a homogeneous polynomial $f\in R=k[z_0, \dots, z_n]$, we say that a matrix factorization is \defi{reduced} if all entries belong to the maximal ideal.  We use $\mf(f)$ to denote the minimal rank of a reduced, homogeneous matrix factorization, and we let $e(f):=\lfloor \tfrac{1}{2}\codim_{R/f}( \operatorname{Sing}(R/f) -2)\rfloor$.  The equality $\rank \coker \phi + \rank \coker \psi = \rank F$ relates the rank of the corresponding MCM modules to that of $F$.  Since our proof involves passing to non-noetherian rings, and we want to avoid discussing MCM modules in that setting, we now express the rank of the corresponding MCM module entirely in terms of the matrix factorization.  As we recall in \S\ref{sec:notation}, if $f$ is irreducible, then $\det(\phi)$ is a scalar multiple of $f^r$ where $r$ is the rank of the corresponding MCM module $\coker \phi$.   We thus define
\[
\mcm(f) := \min\left\{ \begin{matrix} r \text{ where  $(\phi,\psi)$ is a reduced, homogeneous matrix} \\ \text{factorization of $f$ with  }\det(\phi) \text{ a scalar multiple of $f^r$} \end{matrix}\right\}.
\]
Using this notation, the Buchweitz-Greuel-Schreyer Conjecture for homogeneous polynomials $f$ is that $\mcm(f) \geq 2^{e(f)}$; and this implies the weaker conjecture that $\mf(f)\geq 2^{e(f)+1}$.  
While $\mcm(f)$ is the subject of Conjecture~\ref{conj:main}, $\mf(f)$ has arisen in connection with other conjectures (see, for example \cite[\S4]{eps}) and so we will consider both invariants.
Since $\mf(f),e(f)$ and $\mcm(f)$ are all invariant under adjoining extra variables (see Lemma \ref{lem:efmf}), we will consider polynomials in $S:=k[z_0,z_1,\dots]$ where $k$ is an algebraically closed field.  Let $\mathbf S_{d,s}$ be the set of homogeneous degree $d$ and strength $\leq s$ polynomials in $S$.  We will show:

\begin{thm}\label{thm:main}
Fix $d\geq 2$ and $s\geq 1$.  There is a dense open subset $(\mathbf S_{d,s})^\circ \subseteq \mathbf S_{d,s}$ where, for any $f\in (\mathbf S_{d,s})^\circ$, we have that  $\mf(f) \geq 2^{e(f)+1}$ and $\mcm(f) \geq 2^{e(f)}$.
\end{thm}

The main idea behind the proof is a spin on the \defi{Ananyan--Hochster Principle} which was implicitly at the heart of~\cite{ananyan-hochster} and followup works like~\cite{stillman,hartshorne}; see~\cite{bulletin} for an expository discussion of the principle, which we briefly review.  The \defi{collective strength} of homogeneous polynomials $f_1, \dots, f_r \in k[z_0,z_1, \dots ]$ is the minimal strength of a $k$-linear combination of $f_1, \dots, f_r$.  The Ananyan--Hochster Principle says: {\em homogenous polynomials $f_1, \dots, f_r$ of sufficiently high collective strength will behave increasingly like independent variables} $x_1, \dots, x_r$.  We use a variant of this principle based on the following definition:
\begin{defn}\label{defn:secondary strength}
If $f$ is a homogeneous polynomial of strength $s$, then the \defi{secondary strength of $f$} is the maximal collective strength of $g_0, g_1, \dots, g_s, h_0, h_1, \dots, h_s$ where $f=\sum_{i=0}^s g_ih_i$.
\end{defn}

For instance, in characteristic $0$, the power sum $g_{d,n}:=\sum_{i=0}^n z_i^d$ has strength $\geq \frac{n-1}{2}$.\footnote{We do not know the precise strength of $g_{d,n}$ but this inequality follows from two simple computations.  First, $g_{d,n}$ defines a smooth hypersurface in $\PP^n$.  Second, if $f$ is a polynomial of strength $s$, then the corresponding hypersurface is singular in codimension $\leq 2s+1$.} Thus $g_{1,n}g_{6,n}+g_{2,n}g_{5,n}+g_{3,n}g_{4,n}$ has strength $2$ but secondary strength at least $\frac{n-1}{2}$.

Following the logic of the Ananyan--Hochster Principle: as the {\em secondary} strength of $f$ goes to infinity, the $g_i$ and $h_i$ will behave increasingly like independent variables, and so $f$ should behave increasingly like the quadratic form $\sum_{i=0}^s x_iy_i$.  In other words, we propose:
\begin{equation}\label{eqn:secondary strength principle}
\begin{matrix}
\text{{\em A polynomial $f$ of strength $s$ will behave increasingly like}}\\ \text{\em{ a quadratic form of rank $2s+2$ as the secondary strength $\to \infty$.}}
\end{matrix}
\end{equation}
This is the central idea behind Theorem~\ref{thm:main}.    If $q=\sum_{i=0}^s x_iy_i$ then Kn\"orrer's work implies that $\mf(q)=2^s$ and $\mcm(q)=2^{s-1}$.  We prove that the same holds for any $f$ of high enough secondary strength:

\begin{thm}\label{thm:secondary strength}
Fix $d\geq 2$ and $s\geq 1$.  There exists $N=N(d,s)$ where, for any homogeneous polynomial $f$ of degree $d$, strength $s$, and secondary strength $\geq N$, we have that $\mf(f) =  2^s$ and $\mcm(f) = 2^{s-1}$.
 \end{thm}
By noting that $s\geq e(f)+1$ and that the locus in $\mathbf S_{d,s}$ of secondary strength $\geq N$ contains a dense open subset, Theorem~\ref{thm:secondary strength} implies Theorem~\ref{thm:main}.   

The upper bound  $\mf(f)\leq 2^s$ was known.  In fact, a decomposition $f=\sum_{i=0}^s g_ih_i$ yields an explicit matrix factorization of $f$ by specializing the rank $2^s$ matrix factorization of the quadric $\sum_{i=0}^s x_iy_i$.  For instance, if $s=1$ then
\[
\begin{pmatrix}
g_0&-g_1\\
h_1 & h_0
\end{pmatrix}
\text{ and } 
\begin{pmatrix}
h_0&g_1\\
-h_1&g_0\end{pmatrix}
\]
is such a matrix factorization of $f=g_0h_0+g_1h_1$.  This specialization argument is similar to the essential idea behind the construction of matrix factorizations in~\cite[\S2.1]{bgs}.   Via a similar idea, one easily also obtains $\mcm(f)\leq 2^{s-1}$.  So the content of Theorem~\ref{thm:secondary strength} are the inequalities $\mf(f) \leq 2^s$ and $\mcm(f) \leq 2^{s-1}$.

To prove these, we apply ultraproduct machinery from~\cite{stillman}.  The first key observation is that ultraproducts play well with constructing matrix factorizations of a fixed rank and with fixed determinant; therefore invariants like $\mf(f)$ and $\mcm(f)$ behave well under passage to ultraproduct limits.  The second key observation is that, in the ultraproduct limit, an element of strength $s$ and infinite secondary strength will essentially equal a quadric of the form $x_0y_0 +  \cdots + x_sy_s$; this is like the limit of the principle from \eqref{eqn:secondary strength principle}, and it is largely a consequence of the results of~\cite{stillman}. Combining these, Theorem~\ref{thm:secondary strength} is then reduced to Kn{\"o}rrer's result about matrix factorizations of quadric polynomials~\cite{knorrer}

\medskip 

Since Ulrich modules are a special type of maximal Cohen-Macaulay modules, Theorem~\ref{thm:main} also trivially implies that the same lower bound holds for Ulrich modules on the hypersurface ring $k[z_0, \dots, z_n]/f$, or equivalently for Ulrich sheaves on $V(f)$.  In particular, Theorem~\ref{thm:main} shows that the conjecture of Bl\"aser-Eisenbud-Schreyer~\cite[Conjecture~0.3]{blaser-eisenbud-schreyer} on the Ulrich complexity of a hypersurface is also true for generic polynomials.  We add that Ravindra and Tripathi have proven bounds on Ulrich bundles on general hypersurfaces in interesting cases, including degree $d\geq 4$ hypersurfaces in $\mathbb P^4$~\cite{rt3}.

\medskip

There has been a huge array of recent work on the properties of high strength polynomials with applications to algebra and algebraic geometry~\cite{ah2,bbov,bbov2,bv,bdov,bde,bdde,bdes,cmpv, ds, DES,cohom,HS, hartshorne, kz-high-rank, kz-ranks,kp1,kp2,kp3,snowden-relative} as well as connections with arithmetic questions~\cite{birch,cook-magyar,schmidt,xy,y}.\footnote{As strength is an elementary notion with a long history, terminology is inconsistent.  While the term strength originated in~\cite{ananyan-hochster}, it is sometimes known as Schmidt's ``$h$-invariant'' in arithmetic settings; and it predates even Schmidt, appearing at least as early as~\cite{davenport-lewis}.}  By comparison, the central novelty in our proof is to fix the strength but allow the secondary strength to be very large; this raises new Questions like ~\ref{q:unirational} and \ref{q:arithmetic}  

The use of graded ultraproducts of polynomial rings in a fixed number of variables has precursors in~\cite{schoutens, van-den-Dries-schmidt} and more.  Our use of ultraproducts, which requires polynomials in increasing numbers of variables, stems from~\cite{stillman,hartshorne}.  Of course, the use of an ultraproduct means that we obtain no effective bound in Theorem~\ref{thm:secondary strength}.

\begin{remark}
Invariants like $\mf(f)$ and $\mcm(f)$ can only decrease under extension of scalars, so the Buchweitz-Greuel-Schreyer Conjecture for polynomials can be reduced to the algebraically closed case.  But strength and secondary strength can also decrease under extension of scalars.   The algebraically closed assumption is not necessary for the part of Theorem~\ref{thm:secondary strength} on $\mf(f)$; in fact, the proof goes through essentially verbatim if one works with the ultraproduct over $\{k_i[x_1, x_2, \dots ]\}_{i\in \mathbb N}$ where the $k_i$ are arbitrary fields and one applies~\cite[Theorem~6.17]{imperfect}.  For the part on $\mcm(f)$, we do require the algebraically closed condition when reducing to the case where $\deg(\phi) = f^r$ on the nose, and not just up to a scalar.  More substantially, we use the algebraically closed hypothesis in our set-theoretic arguments that the locus in $\mathbf{S}_{d,s}$ of secondary strength $\geq N$ contains an open set, and we thus use that assumption in the passage from Theorem~\ref{thm:secondary strength} to Theorem~\ref{thm:main}.
\end{remark}

\section*{Acknowledgments}  We thank Christine Berkesch, Michael K.~ Brown, David Eisenbud, Jordan Ellenberg, Daniele Faenzi, Yeongrak Kim, Irena Peeva, Joan Pons-Llopis, Claudiu Raicu, R.V. Ravindra, Eric Riedl, Steven V Sam,  Andrew Snowden, Amit Tripathi, and Mark Walker for helpful comments and conversations.

\section{Notation}\label{sec:notation}
We let $k$ be an algebraically closed field of arbitrary characteristic.  Let $S=k[z_0,z_1, \dots]$ with the standard grading and let $\mathfrak m = \langle z_0,z_1,\dots \rangle$ be the homogeneous maximal ideal. We summarize some basic facts about matrix factorizations; details and background can be found in~\cite{eisenbud}.  A \defi{matrix factorization} of $f$ is a pair of graded morphisms $(\phi,\psi) = (\phi\colon F\to G, \psi\colon G\to F)$ where $F$ and $G$ are finite rank free graded $S$-modules, and where $\phi\circ \psi = f\cdot \operatorname{id}_G$ and $\psi\circ \phi =f\cdot \operatorname{id}_F$.  We say that $(\phi,\psi)$ is \defi{reduced} if $\im \phi\subseteq \mathfrak m G$ and $\im \psi \subseteq \mathfrak m F$.  Throughout the remainder of this paper, we focus entirely on reduced matrix factorizations.  We set the \defi{rank of a matrix factorization} to be the rank of $F$, which equals the rank of $G$.
The map $(\phi,\psi) \mapsto \coker \phi$ yields a bijection between reduced matrix factorizations of $f$ and maximal Cohen-Macaulay modules of $S/f$, up to isomorphism.   

The rank of $\coker \phi$ can be determined from properties of the matrix factorization $(\phi,\psi)$.  If $f=gh$ is reducible then $(g)$ and $(h)$ form a rank $1$ matrix factorization, and we will ignore this case from here on.  If $f$ is irreducible then $\det(\phi)$ is a scalar multiple of $f^r$, where $r$ is the rank of $\coker \phi$ over $S/f$; see ~\cite[Remark 2.5(3)]{bgs}.  Since we are over an algebraically closed field, we can rescale $\phi$ so that $\det \phi = f^r$ on the nose, and we will assume this throughout.
\medskip

Our main proof uses the graded ultraproduct ring as defined in~\cite[\S4.1]{stillman}, and we refer the reader to \cite{stillman} for a more detailed discussion or to \cite{bulletin} for an expository treatment.  Let $\mathcal F$ be a non-principal ultrafilter on $\mathbb N$.  Following conventions~\cite{stillman}, we refer to subsets of $\mathcal F$ as {\bf neighborhoods of $\ast$}.  Let $\bR$ denote the graded ultraproduct of $\{k[z_0,z_1,,\dots]\}_{i\in \mathbb N}$, where each polynomial ring is given the standard grading.  
\begin{remark}\label{rmk:up}
The key points about ultrafilters and the ring $\bR$ are the following:
\begin{enumerate}
	\item  Every degree $d$ element  $g\in \bR$ can be represented by a sequence $(g_i)_{i\in \mathbb N}$ of degree $d$ elements $g_i\in S$.  This is a many-to-one relation, i.e. there are many sequences which correspond to the same element.  The equivalence relation is determined by the ultrafilter.  Specifically, if we represent $g$ as $(g_i)$ then we have that $g=0$ if and only if $g_i=0$ for all $i$ in some neighborhood of $\ast$.  
	\item  For any decomposition of $\mathbb N$ into a finite number of disjoint subsets $\mathbb N = \sqcup_{i=1}^r \mathcal N_i$, one of the $\mathcal N_i$ is a neighborhood of $\ast$.
\end{enumerate}
For example, say we have a sequence of elements $(a_i)_{i\in \mathbb N}$ and each $a_i\in S$ is homogeneous of degree $\leq D$.  We can decompose $\mathbb N=\sqcup_{d=0}^D \mathcal N_d$ where $\mathcal N_d$ is the set of indices $i$ such that $\deg(a_i)=d$.  Then one of the $\mathcal N_d$ is a neighborhood of $\ast$.  Thus, for some $d$, we have $\deg(a_i)=d$ for all $i$ in some neighborhood.  We will use this argument several times.
\end{remark}

\section{Basics on $\mathbf S_{d,s}$}\label{sec:topological stuff}
Throughout, $d$ will denote a positive integer of degree $\geq 2$.  We let $\mathbf S_{d}$ be the space of degree $d$ polynomials in $S=k[z_0, z_1, \dots ]$ and let $\mathbf S_{d,s}\subseteq \mathbf S_{d}$ be the sublocus consisting of polynomials of strength $\leq s$.   We will need some basic facts about $\mathbf S_{d,s}$.  We will avoid the more powerful framework of $\GL$-varieties introduced in~\cite{bdes} because those results are (at the moment) only developed in characteristic zero, and because the facts we need are elementary; the interested reader should see~\cite{bdes} as well \cite{draisma} for much more general results on polynomial functors, and~\cite[Theorem~4]{bde} \cite[Theorem 1.9]{kz-high-rank} for results on strength conditions which contain invariant subloci.

Since every element of $S$ of degree $\geq 2$ has finite strength, we have $\mathbf S_d = \bigcup_{s\geq 0} \mathbf S_{d,s}$.  One can also realize $\mathbf S_{d}$ as the direct limit $n\to \infty$ of the affine spaces $\mathbf S^{(n)}_d$ of degree $d$ polynomials in $k[z_0, \dots, z_n]$, with respect to the natural inclusions $k[z_0, \dots, z_n] \subseteq k[z_0, \dots, z_{n+1}]$; we use the Zariski topology on those spaces to induce the direct limit topology on $\mathbf S_d$ and on $\mathbf S_{d,s}$.

\begin{remark}
We choose to the space $\mathbf S_{d}$ of degree $d$ polynomials in $k[z_0,z_1,\dots]$, which is a direct limit of the $\mathbf S_{d}^{(n)}$ as $n\to \infty$.  One can also naturally construct an inverse limit of the $\mathbf S_{d}^{(n)}$, and that would lead to studying degree $d$ elements in the inverse limit of polynomial rings that arises in~\cite[\S1.1]{stillman}.  Since ~\cite[Proposition~2.1]{genstillman} allows one to relate the invariant closed subsets in the inverse and direct limits, there is no significant distinction between these two limits for the results we will be using.
\end{remark}

Polynomials of degree $d$ and strength $\leq s$ come in different flavors depending on the degrees of the polynomials arising in the decomposition.  Let $\mu=(\mu_0, \mu_1, \dots, \mu_s)$ be an integer vector satisfying $1\leq \mu_0\leq \mu_1 \leq \cdots \mu_s \leq \lfloor \frac{d}{2}\rfloor$.  Let $f$ be a homogeneous polynomial of degree $d$; a \defi{strength decomposition of type $\mu$} is an expression $f=\sum_{i=0}^s g_ih_i$ where $\deg(g_i)=\mu_i$.  For instance, the polynomial $f=x_0x_1^5 + x_2x_3^4+x_4^3x_5^3$ admits a strength decomposition of type $\mu = (1,2,3)$. We write $\mathbf S_{d,\mu}$ for the subset of $\mathbf S_{d}$ consisting of degree $d$ polynomials that admit a strength decomposition of type $\mu$.  The locus $\mathbf S_{d,\mu}$ is the image of the map
\begin{equation}\label{eqn:mu}
\pi_{\mu}\colon \mathbf S_{\mu_0} \times \mathbf S_{\mu_1} \times \cdots \times \mathbf S_{\mu_s} \times \mathbf S_{d-\mu_0} \times  \mathbf S_{d-\mu_1} \times \cdots \times \mathbf S_{d-\mu_s} \to \mathbf S_{d}
\end{equation}
where $(g_0,g_1, \dots, g_s,h_0,h_1,\dots,h_s)\mapsto \sum_{i=0}^s g_ih_i$.  In particular, $\mathbf S_{d,\mu}$ is irreducible.  

\begin{lemma}\label{lem:strength stuff}
Fix $d\geq 2$ and $s\geq 0$. Any irreducible component of $\mathbf S_{d,s}$ has the form $ \mathbf S_{d,\mu}$ for some $\mu=(\mu_1, \dots, \mu_s)$ with $1\leq \mu_1 \leq \mu_2 \leq \cdots \leq \mu_s \leq \lfloor \frac{d}{2} \rfloor$.
\end{lemma} 
\begin{proof}
Given $d$ and $s$, there are only finitely many possibilities for $\mu$, and the union of these $\mathbf S_{d,\mu}$ (each of which is irreducible) equals the space of $\mathbf S_{d,s}$.
\end{proof}

\begin{lemma}\label{lem:coll strength N}
Let $d_1,  \dots, d_r$ be any positive integers.  Let $W_{<N}$ be the  sublocus of $\mathbf S_{d_1}\times \cdots \times \mathbf S_{d_r}$ of tuples $(f_1, \dots, f_r)$ of collective strength $<N$.
Then $W_{<N}$ is contained in a proper Zariski closed subset of $\mathbf S_{d_1}\times \cdots \times \mathbf S_{d_r}$.
\end{lemma}
\begin{proof}
Write $W^{(n)}_{<N}$ for the locus of tuples $(f_1, \dots, f_r)$ of degrees $d_1, \dots, d_r$  in $k[z_0, \dots, z_n]$ of collective strength $<N$.  Since collective strength is unchanged under adjoining an extra variable, we have $W_{<N} = \cup_n W^{(n)}_{<N}$.  
Write $\mathbf S_{d_i}^{(n)}$ for the space of polynomials of degree $d_i$ in $k[z_0,\dots, z_n]$.
It suffices to show that $W^{(n)}_{<N}$ belongs to a proper, Zariski closed subset of $\mathbf S_{d_1}^{(n)}\times \cdots \times \mathbf S_{d_r}^{(n)}$ for all $n\gg 0$.

Consider the $r \times (n+1)$ Jacobian matrix with $(i,j)$ entry $\frac{\partial f_i}{\partial z_j}$ and let $J$ be the ideal of $r\times r$ minors of this matrix.  This ideal is invariant under replacing $f_1, \dots, f_r$ by elements which generate the same ideal.  So if $f_1, \dots, f_r$ has collective strength $<N$, then some $k$-linear combination of the $f_i$ has strength $<N$; and so we may assume that $f_1$ has strength $<N$ without altering $J$.  Write $f_1 = \sum_{i=0}^{N-1} g_ih_i$. Each partial derivative of $f_1$ belongs to the ideal generated by the $g_i$ and $h_i$ (by the product rule) and thus $J$ also belongs to that ideal.  Thus $J$ has codimension $\leq 2N$ in $S$, because it belongs to an ideal with $\leq 2N$ generators.

We let $Z\subseteq \mathbf S_{d_1}\times \cdots \times \mathbf S_{d_r}$ be the locus of tuples in $k[z_0, z_1, \dots ]$ where this ideal of Jacobian minors $J$ has codimension $\leq 2N$; and we let $Z^{(n)}  \subseteq \mathbf S_{d_1}^{(n)}\times \cdots \times \mathbf S_{d_r}^{(n)}$ be the locus of tuples in $k[z_0, \dots, z_n]$ where this ideal $J$ has codimension $\leq 2N$.  Since the codimension of the ideal $J$ is unchanged under adjoining an extra variable, we have $Z=\cup_n Z^{(n)}$. 

It thus suffices to prove that $Z^{(n)}$ is Zariski closed and nontrivial for $n\gg 0$. We consider the universal family $\mathcal U\subseteq \mathbf S_{d_1}^{(n)}\times \cdots \times \mathbf S_{d_r}^{(n)}\times \mathbb P^n$ consisting of tuples $(f_1, \dots, f_r)$ of degrees $d_1, \dots, d_r$ in $k[z_0, \dots, z_n]$ and points $x\in \mathbb P^n$ such that $x\in V(f_1, \dots, f_r)$.   We have a projection
\[
p\colon \mathcal U \to \mathbf S_{d_1}^{(n)}\times \cdots \times \mathbf S_{d_r}^{(n)},
\]
and $Z^{(n)}$ is the locus in the base where the fiber has codimension $\leq 2N$.  Since $p$ is proper, $Z^{(n)}$ is closed by upper semicontinuity of fiber dimensions, e.g.~\cite[05F6]{stacks}.  For any $n\geq r$, there exist complete intersections $V(f_1, \dots, f_r)$ which are smooth in codimension $n$ (by the Bertini theorem, for example) and so $Z^{(n)}$ is a proper closed subset whenever $n>2N$.  We conclude that $Z$ is a proper, Zariski closed subset which contains $W_{<N}$.
\end{proof}

Finally we turn to the corollary that we will use to connect Theorems~\ref{thm:secondary strength} and \ref{thm:main}.
\begin{cor}\label{cor:open N}
For any $N, d$ and $\mu$, the subset of $f\in \mathbf S_{d,\mu}$ having secondary strength $\geq N$ contains a dense open subset of $\mathbf S_{d,\mu}$. 
\end{cor}
\begin{proof}
Let $U\subseteq  \mathbf S_{d,\mu}$ be the locus of $f$ of secondary strength $\geq N$.  Let $U'$ be the dense open subset of $\mathbf S_{d_1}\times \cdots \times \mathbf S_{d_r}$ constructed in Lemma~\ref{lem:coll strength N} which misses the locus $W_{<N}$ of collective strength $<N$.  Then $U$ is, by definition, the image of $U'$ under the map $\pi_{\mu}$ from \eqref{eqn:mu}.  So it suffices to check that $\pi_{\mu}(U')$ contains a dense open subset.   As in the previous argument, we can check this after restricting to a polynomial ring involving a finite number of variables.
 
 Let $\mathbf S_d^{(n)}$ as defined in the proof of Lemma~\ref{lem:coll strength N} and let $\mathbf S^{(n)}_{d,\mu}$ be the intersection of $\mathbf S_{d,\mu}$ and $\mathbf S^{(n)}_d$.  In other words, $\mathbf S_{d,\mu}^{(n)}$ is the set of degree $d$ polynomials in $k[z_0, \dots, z_n]$ which admit a strength decomposition of type $\mu$.  Let $U'^{(n)} = U' \cap \left( \mathbf S_{d_1}^{(n)}\times \cdots \times \mathbf S_{d_r}^{(n)}\right)$ and let $\pi_\mu^{(n)}$ be the natural analogue of the map $\pi_\mu$.  We have that $U = \cup_n \pi_{\mu}^{(n)} (U'^{(n)})$ and so it suffices to show that $\pi_{\mu}^{(n)} (U'^{(n)})$ contains a dense open subset of $\mathbf S_{d,\mu}^{(n)}$ for $n\gg 0$.  The proof of Lemma~\ref{lem:coll strength N} showed that $U'^{(n)}$ contains a dense open subset for $n\gg 0$, and so $\pi_{\mu}^{(n)} (U'^{(n)})$ is a dense and constructible subset of $\mathbf S_{d,\mu}^{(n)}$ and thus contains a dense open subset. 
\end{proof}

\section{Proof of Main Results}

\begin{lemma}\label{lem:efmf}
Let $f\in k[z_0, \dots, z_n]$.  Neither $e(f)$ nor $\mf(f)$ nor $\mcm(f)$ changes if we view $f\in S=k[z_0, z_1, \dots]$.
\end{lemma}
\begin{proof}
The extension $k[z_0, \dots, z_n] \subseteq S=k[z_0,z_1, \dots]$ preserves the codimension of ideals; a detailed proof of this fact is in~\cite[Proposition~3.3]{stillman}.  Since $e(f)$ is defined in terms of the codimension of the Jacobian ideal $J=\langle \frac{\partial f}{\partial z_i} | i=0,1,\dots, n\rangle$ it follows that $e(f)$ is invariant upon passage to $S$.  Turning now to $\mf(f)$: if we have a reduced matrix factorization $(\phi,\psi)$ of $f$ over $k[z_0, \dots, z_n]$ then the same matrices yield a reduced matrix factorization of $f$ over $S$.  Conversely, consider a reduced matrix factorization $(\phi,\psi)$ of $f$ over $S$.  Let $(\overline{\phi},\overline{\psi})$ be the images of those matrices under the  quotient map $S\to k[z_0,\dots,z_n]$ where $z_i=0$ for $i>n$.  Then  $(\overline{\phi},\overline{\psi})$ is a reduced matrix factorization of $f$ over $k[z_0, \dots, z_n]$.  It follows $\mf(f)$ is the same over $k[z_0, \dots, z_n]$ or $S$.  Since $\det(\phi)=f^r$ in $S$ if and only if $\det(\overline{\phi})=f^r$ in $k[z_0, \dots, z_n]$, we obtain the desired invariance of $\mcm(f)$.
\end{proof}

\begin{proof}[Proof of Theorem~\ref{thm:secondary strength}]
In the introduction, we noted that we have $\mf(f)\leq 2^s$ and $\mcm(f)\leq 2^{s-1}$; while the essential idea is that we can simply specialize the matrix factorization of the quadric $x_0y_0 + \cdots +x_sy_s$, we provide details here for completeness.  For any $t$, we choose any homogeneous polynomials $g_0,\dots,g_t$ and $h_0, \dots, h_t$ where $\deg(g_i)+\deg(h_i)=d$ for all $i=0,1,\dots, t$; we will build a reduced, homogeneous matrix factorization of $g_0h_0+ \dots + g_th_t$.  For $t=0$ we set $F_0=S$ and $B_0=S(-\deg g_0)$ and we define $\alpha_0=(g_0)$ and $\beta_0=(h_0)$.  We have that $F_0(-d) \overset{\beta_0}{\to} G_0 \overset{\alpha_0}{\to}F_0$ is then a homogeneous, reduced matrix factorization of $g_0h_0$.  For $t>1$, we inductively define $F_t=F_{t-1}\oplus G_{t-1}(\deg h_t)$ and $G_t =G_{t-1} \oplus F_{t-1}(-\deg g_t)$, and we define $\alpha_t$ and $\beta_t$ as 
\[
\bordermatrix{&G_{t-1}&F_{t-1}(-\deg g_t)\cr
F_{t-1}&\alpha_{t-1} & g_t \cdot \id \cr
G_{t-1}(\deg h_t)&h_t \cdot \id & -\beta_{t-1}\cr
}
\ \ \text{ and }\ \ 
\bordermatrix{&F_{t-1}&G_{t-1}(-\deg g_t)\cr
G_{t-1}&\beta_{t-1} & g_t \cdot \id \cr
F_{t-1}(\deg h_t)&h_t \cdot \id & -\alpha_{t-1}\cr
}.
\]
By induction, one verifies that: the product of $\alpha_t$ and $\beta_t$ is $g_0h_0+ \dots + g_th_t$ times the identity matrix; the above matrices induce homogeneous maps; and the ranks of $F_t$ and $G_t$ are $2^t$.  Since each nonzero entry of $\alpha_t$ or $\beta_t$ is one of the $g$'s or $h$'s, this matrix factorization is also reduced, and we have thus produced a homogeneous, reduced matrix factorization of $g_0h_0+ \dots + g_th_t$ of rank $2^t$ for all $t$.  This shows that if $f$ has strength $s$, then $\mf(f)$ is at most $2^s$.  Moreover, since $\det(\alpha_t)=\det(\beta_t)$ and $\det(\alpha_t\beta_t)=f^{\rank F_t}=f^{2^t}$ we conclude that $\det(\alpha_t)=f^{2^{t-1}}$.  It follows that if $f$ has strength $s$ then $\mcm(f)$ is at most $2^{s-1}$.

We now turn to showing that there is some $N$ where the following holds: if the collective strength of $g_0,\dots,g_s,h_0,\dots,h_s$  is at least $N$, then $\mf(f)\geq 2^s$.  Assume for contradiction that this is not the case.  Then we can choose a sequence of collections of polynomials $\mathcal P_j :=(g_{0,j}, \dots g_{s,j}, h_{0,j}, \dots h_{s,j})$ with $j\in \mathbb N$ where:
\begin{itemize}
	\item The collective strength of $\mathcal P_j$ goes to infinity as $j\to \infty$ and
	\item $f_j :=\sum_{i=0}^s g_{i,j}h_{i,j}$ has $\mf(f_j)<2^s$.
\end{itemize}
We let $F,G_0, \dots, G_s, H_0, \dots, H_s$ denote the elements in the ultraproduct ring $\bR$ corresponding to the sequences $(f_j), (g_{0,j}), \dots (g_{s,j}), (h_{0,j}), \dots, (h_{s,j})$.  Since there are only finitely many positive integers $m<2^s$, by Remark~\ref{rmk:up}(2), after passing to a neighborhood of $\ast$, we can assume that $\mf(f_j)=m<2^s$ for some fixed $m$ and all $j$ in the chosen neighborhood of $\ast$.  For each $j$ in this neighborhood, we choose any reduced matrix factorization $(\phi_j,\psi_j)$ of $f_j$ of rank $m$.  
Each of the entries of $\phi_j$ and $\psi_j$ must be homogeneous of degree at least $1$ and most $d-1$ (because we started with reduced matrix factorizations); thus by again applying Remark~\ref{rmk:up}(2), after passing to a neighborhood of $\ast$, we can assume that the degree of each entry of $\phi_j$ and $\psi_j$ is indepdendent of $j$.

Now we get to the key point: we can take an ultraproduct of these matrix factorizations $(\phi_j,\psi_j)$ of $f_j$ to obtain a matrix factorization of $F$.  Specifically, we let $\Phi=(\phi_j)$ and $\Psi=(\psi_j)$ be the $m\times m$ matrices whose entries are the ultraproducts of the entries of the $\phi_j$ and $\psi_j$ matrices.  Since $\phi_j\cdot \psi_j = f_j\cdot \operatorname{id}_{S^m}=\psi_j\cdot \phi_j $ for all $j$ in our neighborhood, these equations pass to the ultraproduct and we have $\Phi\cdot \Psi = F\cdot \operatorname{id}_{\bR^m}=\Psi\cdot \Phi.$  In particular, $(\Phi,\Psi)$ is a homogeneous, reduced matrix factorization of $F$ of rank $m<2^s$.

We now produce the contradiction.  The ultrapower ring $\bR$ is a polynomial ring by~\cite[Theorem~1.3]{stillman}. Moreoever, since the collective strength of $\mathcal P_j$ went to infinity, the elements $G_0, \dots, G_s, H_0, \dots, H_s$ have infinite collective strength.  These elements are thus abstract variables in the polynomial ring $\bR$ by~\cite[Theorem~1.3]{stillman}, and we have $F = G_0H_0 + \cdots +G_sH_s$.  The degree of $G_i$ is $\mu_i$ and the degree of $H_i$ is $d-\mu_i$.  But the classification of matrix factorizations of quadrics from~\cite{knorrer} does not make use of the grading, and so Kn{\"o}rrer's classification (when the characteristic of $k$ is not $2$) implies that $\mf(F) \geq 2^s$.  When the characteristic of $k$ is $2$, one can instead apply~\cite[Proposition 3.2]{beh} to obtain $\mf(F)\geq 2^s$.  We have thus contradicted the conclusion of the previous paragraph.

For the statement about $\mcm(f)$, the argument will be quite similar.  We must show that there is some $N$ such that if the collective strength of $g_0,\dots,g_s,h_0,\dots,h_s$  is at least $N$, then $\mcm(f)\geq 2^{s-1}$.  Assume that this is not the case.  Then we can choose a sequence of collections of polynomials $\mathcal P_j :=(g_{0,j}, \dots g_{s,j}, h_{0,j}, \dots h_{s,j})$ with $j\in \mathbb N$ where:
\begin{itemize}
	\item The collective strength of $\mathcal P_j$ goes to infinity as $j\to \infty$ and
	\item $f_j :=\sum_{i=0}^s g_{i,j}h_{i,j}$ has $\mcm(f_j)<2^{s-1}$.
\end{itemize}
The latter condition implies that, for each $j$, there is a matrix factorization $(\phi_j, \psi_j)$ of $f_j$ where $\det(\phi_j)=f_j^{r_j}$ and $r_j<2^{s-1}$.  From here, the argument is nearly identical to the one above: by passing to a neighborhood of $\ast$, we can produce homogeneous matrices $(\Phi, \Psi)$ over $\bR$ where $\det(\Phi)=F^r$ for some $r<2^{s-1}$.  But $F$ is a quadric in the polynomial ring $\bR$, and this would contradict the known results for quadrics~\cite{knorrer,beh}.
\end{proof}

\begin{proof}[Proof of Theorem~\ref{thm:main}]
By Lemma~\ref{lem:strength stuff} it suffices to prove the claim for an open subset of $\mathbf S_{d,\mu}$ for each $\mu=(\mu_0, \dots, \mu_s)$ satisfying $1\leq \mu_0\leq \mu_1 \leq \cdots \mu_s \leq \lfloor \frac{d}{2}\rfloor$.    In other words, we may assume that $f$ has a decomposition $f=\sum_{i=0}^s g_ih_i$ with $\deg(g_i)=\mu_i$ and prove that the claim holds for a generic such $f$.

By a simple application of the product rule, the Jacobian ideal $\langle \frac{\partial f}{\partial z_i} | i = 0, 1, \dots \rangle$ belongs to the ideal $\langle g_0, \dots, g_s, h_0, \dots, h_s\rangle$; this ideal has codimension $\leq 2s+2$ in $S$ by the Principal Ideal Theorem, and thus it has codimenion $\leq 2s+1$ in $S/f$.   It follows that the codimension of the singular locus of $S/f$  in $S/f$ is at most $2s+1$ and thus that $s \geq e(f)+1$. It therefore suffices to prove $\mf(f)\geq 2^s$.

Theorem~\ref{thm:secondary strength} implies that there is some $N$ such that: $\mf(f) \geq 2^s$ and $\mcm(f)\geq 2^{s-1}$ for any $f\in \mathbf S_{d,\mu}$ where the collective strength of $g_0,\dots,g_s,h_0,\dots,h_s$ is at least $N$, and Corollary~\ref{cor:open N} implies the locus of secondary strength $\geq N$ in $\mathbf S_{d,\mu}$ contains a dense open subset.   Since $s\geq e(f)+1$, this implies the desired result.
\end{proof}

\begin{example}\label{ex:all d and e}
Theorem~\ref{thm:secondary strength} implies that, for any given $d$ and $e$, there is a polynomial $F$ of degree $d$ and $e(F)=e$ where the Buchweitz-Greuel-Schreyer Conjecture holds and is sharp.  Let $F=\sum_{i=0}^s x_ig_i$ where each $g_i\in k[y_{i,0}, y_{i,1}, \cdots, y_{i,N}]$ is a generic polynomial of degree $d-1$ for some $N\gg 0$.  
Since the $x$'s and the $g$'s all involve different set of variables, we can assume that the secondary strength of $F$ is at least the strength of $g_0$, which is at least $\frac{N-1}{2}$ because $V(g_i)\subseteq \PP^N$ is smooth.   So Theorem~\ref{thm:secondary strength} applies to $F$ for $N\gg 0$.  We always have $s\leq e(F)+1$ and we must check that $s=e(F)+1$.  We omit the details, but the basic point is that the Jacobian ideal of $F$ contains $g_0, \dots, g_s$ and--after taking ideal quotients with respect to the Jacobian ideals of $g_0, \dots, g_s$--it contains $x_0, \dots, x_s$ as well. To give a specific example in characteristic $0$, we could let $g_i = y_{i,0}^{d-1} + \cdots + y_{i,N}^{d-1}$ for $0\leq i \leq s$.
\end{example}

Theorem~\ref{thm:main} has an implication for arithmetically Cohen-Macaulay vector bundles on hypersurfaces, as noted in~\cite[Conjecture B]{bgs}.
\begin{cor}
There exists $N$ depending on $d$ and $s$ such that: if $f\in k[z_0, \dots, z_n]$ has degree $d$, strength $s$, and secondary strength $\geq N$, then any vector bundle $\mathcal E$ of rank $<2^{s-1}$ on the hypersurface $V(f)\subseteq \mathbb P^n$ which satsifies 
\[
H^i(X,\mathcal E(e)) = 0 \text{ for all } 0<i<n-1 \text{ and all } e\in \ZZ
\]
must be a direct sum of line bundles $\mathcal E = \oplus_j \mathcal O_X(a_j)$.
\end{cor}
See also~\cite{mrr1,mrr2,rt1,rt2,rt3,tripathi} where a number of cases of ~\cite[Conjecture B]{bgs} are shown for hypersurfaces in $\mathbb P^n$ with $n\leq 6$.

\begin{remark}\label{r:mf}
 Theorem~\ref{thm:secondary strength} shows that strength determines $\mf(f)$ and $\mcm(f)$ for generic $f$, but
Faenzi and Pons--Llopis have pointed out that this cannot hold for abitrary $f$.  For instance, they noted the Pfaffian of a general skew-symmetric $6\times 6$ matrix of linear forms, in $\geq 6$ variables, will have strength $\geq 3$; but the skew-symmetric matrix is part of a rank $6$ matrix factorization, and $6 < 2^3$.  Similarly, the cokernel of the skew-symmetric matrix is a rank $2<2^{3-1}$ MCM module (an Ulrich module, in fact) on the hypersurface.  

This provides a counterexample to a question from a prior draft of this paper and raises a question: if $f\in R =k[x_{1,1}, \dots, x_{n,n}]$ is the determinant of a generic $n\times n$ matrix of linear forms, then $f$ has a matrix factorization of rank $n$, namely, the matrix $(x_{i,j})$ and its adjoint. By Laplace expansion along the top row, one can see that the strength of $f$ is at most $n-1$.  Is the strength equal to $n$?  If so, this would show that $\mf(f)$ and $\mcm(f)$ can be arbitrarily far away from $2^{\operatorname{strength}(f)}$.  A similar question can be asked about Pfaffians.
\end{remark}

The proof technique  raises some new geometric and arithmetic directions to consider.  For instance, if the strength of a polynomial $f$ is sufficiently large, then it is known that $V(f)$ is unirational; this follows from the main result of~\cite{hmp} and various results relating strength with the codimension of the singular locus c.f.~\cite{ananyan-hochster,kz-ranks}.  But what happens if the strength is small--so that $V(f)$ is highly singular--but $f$ has large secondary strength?

\begin{question}\label{q:unirational}
Let $f$ be a degree $d$, strength $s$ polynomial.  If the secondary strength of $f$ is sufficiently large relative to $s,d$, will $V(f)$ be unirational?  For a specific example, consider
\[
f=(x_0 + \cdots +x_n)(x_0^4+\cdots +x_n^4) + (x_0^2+\cdots +x_n^2)(x_0^3+\cdots + x_n^3).
\]
The hypersurface $V(f)$ is singular in codimension $4$ and $f$ has strength $1$, but its secondary strength will go to $\infty$ as $n\to \infty$. If $n\gg 0$, will $V(f)\subseteq \mathbb P^{n-1}$ be unirational?
\end{question}
One can also ask arithmetic questions in this vein.
\begin{question}\label{q:arithmetic}
Let $f\in \mathbb Q[z_0, \dots, z_n]$ be a homogeneous polynomial of degree $d$ and strength $s$ and sufficiently high secondary strength.  Does $V(f)$ have nice arithmetic properties, similar to those of high strength polynomials as observed in~\cite{birch,cook-magyar,schmidt,xy,y}?  For example: does $V(f)$ satisfy the Hasse principle?
\end{question}

\end{document}